\documentclass{amsart}

\usepackage{xypic}
\usepackage{amssymb}
\usepackage{dsfont}

\usepackage{lipsum}

\usepackage{amsmath}
\usepackage{url}

\usepackage{tikz-cd}

\xyoption{all}

\newif\iffurther
\furthertrue

\numberwithin{equation}{section}
\numberwithin{figure}{section}
\theoremstyle{plain}
\newtheorem{thm}{Theorem}[section] 
\newtheorem*{thm*}{Theorem}

\newtheorem{prop}[thm]{Proposition}
\newtheorem{cor}[thm]{Corollary}
\theoremstyle{definition}

\theoremstyle{remark}

\newtheorem*{acknowledgement*}{Acknowledgement}

\newcommand\suchthat{\;\ifnum\currentgrouptype=16 \middle\fi|\;}

\def\E{\mathcal{E}}
\def\P{\mathcal{P}}
\def\T{\mathcal{T}}

\begin{document}

\title{Modules over Fomin-Kirillov algebras and their subalgebras}

\author{Be'eri Greenfeld}
\address{Department of Mathematics, University of Washington, Seattle, WA, 98195, USA}
\email{grnfld@uw.edu}

\author{Sarah Mathison}
\address{Department of Mathematics, University of Washington, Seattle, WA, 98195, USA}
\email{smathi2@uw.edu}

\author{Aditya Saini}
\address{Department of Mathematics, University of California San Diego, La Jolla, CA, 92093, USA}
\email{asaini@ucsd.edu}

\author{Scott Wynn}
\address{Department of Mathematics and Department of Computer Science, University of Washington, Seattle, WA, 98195, USA}
\email{scotthw@cs.washington.edu}

\begin{abstract}
We compute the truncated point schemes of subalgebras of Fomin-Kirillov algebras associated with certain graphs. While Fomin-Kirillov algebras do not admit any truncated point modules, we prove a tight bound on the degrees of truncated point modules over generalized Fomin-Kirillov algebras associated with trees.
\end{abstract}

\maketitle

\section{Introduction}

Fomin-Kirillov algebras $\E_n,\ n=3,4,5,\dots$, introduced by Fomin and Kirillov in 1999 in their study of the cohomology of flag manifolds \cite{FK}, are noncommutative quadratic algebras which are intimately related to cohomology theory and homological algebra, algebraic combinatorics, Schubert calculus, Nichols algebras, Hopf algebras, and more \cite{Andruskiewitsch2003,GeneralizedFK,Fomin2000,HV2019,Kirillov2010,Lenart,Meszaros2014,Milinski2000,Stefan2016,WZ}. 
Despite the deep and multidisciplinary study of these algebras, several fundamental problems regarding their structure remained open. One of the most intriguing problems is: which Fomin-Kirillov algebra are finite-dimensional? The answer to this problem is known only for $n\leq 5$, whose (finite) dimensions and Hilbert series have been computed \cite{FK, Vendramin}.

Let $k$ be an algebraically closed base field. Let $A=k\oplus A_1\oplus A_2\oplus \cdots$ be a connected graded $k$-algebra generated in degree one. A point module over $A$ is a graded, cyclic $A$-module $M=M_0\oplus M_1\oplus \dots$ generated in degree $0$ with $1$-dimensional homogeneous components; that is, with Hilbert series $H_M(t)=1+t+t^2+\cdots$. Point modules of noncommutative graded rings resemble the role of points on projective algebraic varieties. Indeed, over commutative algebra, point modules are parametrized by the `proj' scheme. An important related notion is of truncated point modules.
A degree-$n$ truncated (or `$n$-truncated') point module is a graded cyclic module with Hilbert series $1+t+\cdots+t^n$. We denote by $\P_n(A)$ the space of $n$-truncated point modules of $A$. The limit $\varprojlim_{n} \mathcal{P}_n(A)$ is the space of point modules. 
Point modules and truncated point modules have thus become fundamental in noncommutative projective algebraic geometry, enabling one to attach a geometric object to a given noncommutative graded algebra, see \cite{ATV2, ATV, AZ, AZ2, RRZ, RZ, Smith} and references therein. Ring-theoretic properties can be studied using canonical maps to twisted rings corresponding to the moduli spaces of point modules, see \cite{RRZ, RZ}. Notably, in \cite{SW}, Sierra and Walton solved a long-standing open problem on the (non-)Noetherianity of certain enveloping algebras using morphisms into twisted coordinate rings of their spaces of point modules. It is also worth mentioning that even some very well behaved graded algebras can lack point modules, see \cite{Vancliff2024}.

In this paper, motivated by the aforementioned results and applications, we study the spaces of (truncated) point modules over Fomin-Kirillov algebras and their graded subalgebras. Generalized Fomin-Kirillov algebras, defined and studied in \cite{GeneralizedFK}, are subalgebras $\E_G\subseteq \E_n$ generated by the edges of subgraphs $G\subseteq K_n$. Understanding the space of (truncated) point modules of such algebras can be extremely useful. Most elementary, if $\P_d(B)\neq \emptyset$ for some graded subalgebra $B\subseteq A$, it follows that $A_d\neq 0$, and if $B$ admits a point module then both $A,B$ are infinite-dimensional. It was proposed to study point modules of Nichols algebras in \cite{An2004}, and this was done in \cite{ABFF2023} for an interesting class of Nichols algebras. The intimate connection between Fomin-Kirillov algebras and Nichols algebras serves as another motivation for our work. 

The main theme of this paper is that truncated point modules of generalized Fomin-Kirillov algebras are as restricted as possible. First:

\begin{thm} \label{thm:fk}
Let $n\geq 3$ be an integer. Then $\mathcal{E}_n$ admits no truncated point modules of degree greater than $1$.
\end{thm}
For a graded algebra $A$ one can measure the optimal degree of truncated point modules of $A$ by $p(A):=\sup\{i\in \mathbb{N}\ |\ \P_i(A) \neq \emptyset\}$. 
Thus, $p(\E_n)=1$, the smallest possible value\footnote{Notice that $\P_1(A)\cong \mathbb{P}(A_1)$.}.  

We next compute $\P_*(\E_G)$ for some graphs $G \subseteq K_n$; the main difficulty is that generalized Fomin-Kirillov algebras are usually not quadratic, and no general useful presentation of them is known. We prove the following:

\begin{thm} \label{thm:trees}
For every tree $\T$ we have that $p(\E_\T)$ is at most the number of edges of $\T$, 
and this bound is best possible.
\end{thm}
\noindent For instance, $p(\E_{A_n})=n-1$ for the Dynkin diagram $A_n$ with $n-1$ edges (Proposition \ref{prop:An}).

\bigskip

\noindent \emph{Conventions and notations}: All of the algebras in this paper are connected graded (that is, their degree-$0$ part is the base field) and generated by their degree-$1$ homogeneous components. We denote by $K_n$ the complete graph on $n$ vertices and by $K_{n,m}$ the complete bipartite graph on $n+m$ vertices. For an $n\times m$ matrix $M$ we denote by $M(i_1,\dots,i_k)$ its $k\times m$ sub-matrix consisting of rows $i_1,\dots,i_k$.

\section*{Acknowledgements}

The authors thank James Zhang for interesting related discussions. This project has originated from the Washington eXperimental Math Lab (WXML) at the University of Washington under the mentorship of the first-named author.

\section{Modules over Fomin-Kirillov algebras}

Let $A=k\left<x_1,\dots,x_d\right>/\left<f_1,f_2,\dots\right>$ be a graded algebra. For an $n$-truncated point module $M$ over $A$ we can fix a homogeneous basis $M=ke_0\oplus \cdots \oplus ke_n$ and encode the action of $A$ on $M$ by constants $\lambda_i^{j}\in k,\ 1\leq i\leq d,\ 0\leq j\leq n-1$ such that:
\[
x_i \cdot e_j = \lambda_i^{j} e_{j+1}.
\]
Thus $([\lambda_1^{0}\colon \cdots \colon \lambda_d^{0}],\dots,[\lambda_1^{n-1}\colon \cdots \colon \lambda_d^{n-1}])\in \mathbb{P}^{d-1}\times \cdots \times \mathbb{P}^{d-1}$. Those points in the product of projective spaces corresponding to truncated point modules can be found as the solutions of a concrete system of homogeneous equations derived from the presentation $k\left<x_1,\dots,x_d\right>/\left<f_1,f_2,\dots\right>$ by multilinearization, namely, each relation $f=\sum_I c_Ix_{i_1}\cdots x_{i_t}$ induces the polynomial equations $\sum_I c_I \lambda_{i_1}^{s+t-1}\cdots \lambda_{i_t}^s=0$ for each $0\leq s\leq n-t$. The set of points in $\left(\mathbb{P}^{d-1}\right)^{\times n}$ solving these multilinearized polynomial equations is in bijection with $\P_n(A)$, the set of $n$-truncated point modules over $A$; see \cite[Proposition 3.9]{ATV} and \cite{Rogalski}.

Let us focus on the Fomin-Kirillov algebra $\E_n$ for some $n$. It has the following presentation: $\E_n$ is generated by $\{x_{ij}\ |\ 1\leq i<j\leq n\}$ (all of degree $1$), subject to the relations:
\begin{eqnarray*}
  &(R^1_{ij})&   x_{ij}^2 = 0\ \ \text{for all}\ i<j \\
   &(R^2_{ij,kl})& x_{ij}x_{kl} - x_{kl}x_{ij} = 0 \ \ \text{for all}\ i<j,\ k<l  \\
   &(R^3_{ijk})& x_{ij}x_{jk} - x_{jk}x_{ik} - x_{ik}x_{ij} = 0\ \ \text{for all}\ i<j<k  \\
   &(R^4_{ijk})& x_{jk}x_{ij} - x_{ik}x_{jk} - x_{ij}x_{ik} = 0\ \ \text{for all}\ i<j<k.
\end{eqnarray*}
In the space of $2$-truncated point modules:
\[
\mathcal{P}_2(\E_n) \subseteq \mathbb{P}^{{n \choose 2} - 1} \times \mathbb{P}^{{n \choose 2} - 1}
\]
with homogeneous coordinates $([x_{12}\colon\cdots\colon x_{n-1\ n}],[y_{12}\colon\cdots\colon y_{n-1\ n}])$, we have the following polynomial equations:

\begin{eqnarray*}
  &(E^1_{ij})&   y_{ij}x_{ij} = 0\ \ \text{for all}\ i<j \\
   &(E^2_{ij,kl})& y_{ij}x_{kl} - y_{kl}x_{ij} = 0 \ \ \text{for all}\ i<j,\ k<l  \\
   &(E^3_{ijk})& y_{ij}x_{jk} - y_{jk}x_{ik} - y_{ik}x_{ij} = 0\ \ \text{for all}\ i<j<k  \\
   &(E^4_{ijk})& y_{jk}x_{ij} - y_{ik}x_{jk} - y_{ij}x_{ik} = 0\ \ \text{for all}\ i<j<k. 
\end{eqnarray*}

\begin{prop}\label{prop:E3}
The algebra $\E_3$ admits no truncated point modules of degree greater than $1$.
\end{prop}

\begin{proof}
Consider:
\[
\mathcal{P}_2(\E_3) \subseteq \mathbb{P}^2_{x_{12},x_{13},x_{23}} \times \mathbb{P}^2_{y_{12},y_{13},y_{23}}
\]
and organize the defining polynomial equations in a matrix form:
\[
\begin{matrix}
(E^3_{123}) \\    
(E^4_{123}) \\    
(E^1_{12}) \\    
(E^1_{13}) \\    
(E^1_{23})
\end{matrix}
\ \ \ 
\underbrace{\left(\begin{matrix}
x_{23} & -x_{12} & -x_{13} \\
-x_{13} & -x_{23} & x_{12} \\
x_{12} & 0 & 0 \\
0 & x_{13} & 0 \\
0 & 0 & x_{23}
\end{matrix}\right)}_{M}
\left(\begin{matrix}
y_{12} \\
y_{13} \\
y_{23} \\
\end{matrix}\right) = \left(\begin{matrix}0 \\ 0 \\ 0\\ 0\\ 0 \end{matrix}\right).
\]
If $x_{12},x_{13},x_{23}$ are all non-zero then $M(3,4,5)$ is invertible. If exactly one of $x_{12},x_{13},x_{23}$ is zero then one of $M(1,4,5),M(1,3,5),M(1,3,4)$ is invertible (respectively). 
If exactly one of $x_{12},x_{13},x_{23}$ is non-zero then one of $M(1,2,3),M(1,2,4),M(1,2,5)$ is invertible (respectively). 
In any case, we see that $\text{rank}(M)=3$, so the only solution to the homogeneous equation $M\cdot \vec{y}=\vec{0}$ is the trivial solution, hence $y_{12}=y_{13}=y_{23}=0$. This shows that $\P_2(\E_3)=\emptyset$.
\end{proof}

\begin{prop}\label{prop:E4}
The algebra $\E_4$ admits no truncated point modules of degree greater than $1$.
\end{prop}

\begin{proof}
Consider:
\[
\mathcal{P}_2(\E_4) \subseteq \mathbb{P}^5_{x_{12},x_{13},x_{14},x_{23},x_{24},x_{34}} \times \mathbb{P}^5_{y_{12},y_{13},y_{14},y_{23},y_{24},y_{34}}
\]
and assume that there is some point ${\bf x} = ([x_{12}\colon \cdots \colon x_{34}],[y_{12}\colon \cdots \colon y_{34}])$ in $\P_2(\E_4)$. We may assume that $x_{12}\neq 0$, since at least one of the coordinates $x_{ij}$ is non-zero and $\text{Aut}(\E_n)$ acts transitively on the generators. 
Furthermore, we may assume that $x_{12} = 1$. Now by $(E^1_{12})$, we have that $y_{12}=0$. By $(E^2_{12,34})$, we have $y_{12}x_{34} = y_{34}x_{12}$ so $y_{34}=0$. The following homogeneous system holds:

\[
\begin{matrix}
    (E^3_{123}) \\
    (E^4_{123}) \\
    (E^3_{124}) \\
    (E^4_{124}) \\
    (E^2_{13,24}) \\
    (E^1_{14}) \\
    (E^1_{23}) \\
    (E^1_{24})
\end{matrix}
\ \ \ 
\underbrace{\left(\begin{matrix}
1 & 0 & x_{13} & 0 \\
-x_{23} & 0 & 1 & 0 \\
0 & 1 & 0 & x_{14} \\
0 & -x_{24} & 0 & 1 \\
x_{24} & 0 & 0 & -x_{13} \\
0 & x_{14} & 0 & 0 \\
0 & 0 & x_{23} & 0 \\
0 & 0 & 0 & x_{24}
\end{matrix}\right)}_{M}
\left(\begin{matrix}
y_{13} \\
y_{14} \\
y_{23} \\
y_{24}
\end{matrix}\right) = \left(\begin{matrix}0 \\ 0 \\ 0 \\ 0\\ 0 \\ 0 \\ 0 \\ 0\end{matrix}\right).
\]
Now $M$ is row equivalent (by adding a multiple of the 1st row to the 2nd row and a multiple of the 3rd row to the 4th row) to:
\[
M' = \left(\begin{matrix}
1 & 0 & x_{13} & 0 \\
0 & 0 & 1+x_{13}x_{23} & 0 \\
0 & 1 & 0 & x_{14} \\
0 & 0 & 0 & 1+x_{14}x_{24} \\
x_{24} & 0 & 0 & -x_{13} \\
0 & x_{14} & 0 & 0 \\
0 & 0 & x_{23} & 0 \\
0 & 0 & 0 & x_{24}
\end{matrix}\right).
\]
If both $1+x_{13}x_{23},1+x_{14}x_{24}$ are non-zero then $M'(1,2,3,4)$ is invertible. Assume otherwise. If $1+x_{13}x_{23}=0$ then in particular $x_{13},x_{23}\neq 0$, and $M'(1,3,5,7)$ is invertible. If $1+x_{13}x_{23}\neq 0$ but $1+x_{14}x_{24}=0$ then in particular $x_{14},x_{24}\neq 0$, and $M'(1,2,3,8)$ is invertible. It follows that in any case, $\text{rank}(M)=\text{rank}(M')=4$ and therefore the only solution to the homogeneous equation $M\cdot \vec{y}=\vec{0}$ is when $\vec{y}=0$, so $y_{12}=y_{13}=y_{14}=y_{23}=y_{24}=y_{34}=0$. It follows that $\P_2(\E_4)=\emptyset$.
\end{proof}

\begin{proof}[{Proof of Theorem \ref{thm:fk}}]
It suffices to prove the claim for truncated point modules of degree $2$, since we have a sequence of morphisms given by truncation: \[ \cdots \rightarrow \P_3(\E_n) \rightarrow \P_2(\E_n). \] Let $n\geq 3$ and consider a truncated point module $M$ over $\mathcal{E}_n$ of degree $2$ and fix a homogeneous basis, say, $M=ke_0\oplus ke_1 \oplus ke_2$. Pick arbitrary distinct generators $x_{ij},x_{kl}\in \E_n$ (the indices $i,j,k,l$ need not be distinct). Consider the subalgebra $A\subseteq \E_n$ generated by $\{x_{pq}|p,q\in \{i,j,k,l\}\}$. Notice that $A\cong \E_3$ (if $|\{i,j,k,l\}|=3$) or $A\cong \E_4$ (if $|\{i,j,k,l\}|=4$) as graded algebras. Indeed, if at least one of $r,s,t,u$ is not in $\{i,j,k,l\}$ then the relations $(R^{1}_{rs}),(R^{2}_{rs,tu}),(R^{3}_{rst}),(R^{4}_{rst})$ all become trivial modulo the ideal generated by $\{x_{pq}|p\notin \{i,j,k,l\}\ \text{or}\ q\in \{i,j,k,l\}\}$.

By restriction, $M$ is also a graded $A$-module. By Propositions \ref{prop:E3},\ref{prop:E4}, $M$ cannot be generated by $e_0$. Therefore, either $A_1\cdot e_0=0$ or $A_1\cdot e_1=0$, so, since $A$ is generated in degree $1$, $A_2\cdot e_0=(A_1)^2\cdot e_0=0$. In particular, $x_{ij}x_{kl}\cdot e_0=0$. Since $x_{ij},x_{kl}$ were arbitrary (distinct) generators of $\E_n$ (and all of the generators are squared zero), it follows that $(\E_n)_2\cdot e_0 = 0$, so $M$ is not generated by its degree-$0$ component. Therefore $\E_n$ admits no truncated point modules of degree $\geq 2$.
\end{proof}

\section{Generalized Fomin-Kirillov Algebras}

\begin{cor}
Let $G$ be a graph on $n$ vertices. Then $\mathcal{P}_2(\E_G)=\emptyset$ if and only if $G$ is a disjoint union of complete graphs.
\end{cor}

\begin{proof}
Let $G = \bigsqcup_{i=1}^{m} K_{n_i}$ be a disjoint union of complete graphs. If there is a truncated point module $M = ke_0\oplus ke_1 \oplus ke_2$ over $\E_G$ --- and hence over $\bigotimes_{i=1}^{m} \E_{n_i}$ --- then for some $1\leq i,j\leq m$, we have $x_{pq}\in \E_{K_{n_i}}$ and $x_{rs}\in \E_{K_{n_j}}$ such that $x_{pq}\cdot e_0 \neq 0,x_{rs}\cdot e_1\neq 0$. If $i=j$, we obtain a contradiction to Theorem \ref{thm:fk}, since then $M$ becomes a degree-$2$ point module over $\E_{n_i}$. If $i\neq j$ then $0\neq x_{rs}x_{pq}\cdot e_0 = x_{pq}\cdot(x_{rs}\cdot e_0)$, but $x_{rs}\cdot e_0=0$ (for otherwise $x_{rs}^2\cdot e_0$ is a non-zero scalar multiple of $e_2$, contradicting that $x_{rs}^2=0$), a contradiction.

Conversely, assume that $G$ is not a disjoint union of complete graphs; pick a connected component which is not a complete graph. We can find distinct vertices $p,q,r$ such that the edges $(p,q),(q,r)$ are in $G$ but $(p,r)$ is not (it follows by induction that a finite connected graph G such that whenever $(p,q),(q,r)$ are in G, $(p,r)$ also is in G, is a complete graph). 
Thus the subgraph of $G$ consisting of the vertices $p,q,r$ is $A_3$ (this is a graph with three vertices and two edges) and by \cite[Example 2.2 \ \text{and}\ Theorem 6.1]{GeneralizedFK}, $\E_{A_3}\cong k\left<a,b\right>/\left<a^2,b^2,aba-bab\right>$. Consider the left ideal $L = \left<a\right>\leq \E_{A_3}$ and observe that $\E_{A_3}/L\cong k \oplus kb \oplus kab$ is a truncated point module of degree $2$. Notice that $\E_G\twoheadrightarrow \E_{A_3}$, moding out by all edges incidenting with a vertex outside $\{p,q,r\}$. By inflation, we obtain a $2$-truncated point module over $\E_G$, and it follows that $\P_2(\E_G)\neq \emptyset$.
\end{proof}

\begin{prop} \label{prop:An}
Let $n\geq 3$ and let $A_n$ be the graph with vertices $\{1,\dots,n\}$ and with edges $(i,i+1)$ for $1\leq i\leq n-1$. Then:
$$ \# \mathcal{P}_d(\mathcal{E}_{A_n}) = \begin{cases} 
      2n-2d & \text{if}\ 3\leq d\leq n-1 \\
      0 & \text{if}\  d\geq n
   \end{cases}
$$
and $\P_2(\E_{A_n})$ is a union of $2n-6$ copies of $\mathbb{P}^1$ if $n>3$, and if $n=3$, two points.
\end{prop}

\begin{proof}
By \cite[Theorem 6.1]{GeneralizedFK}, $\E_{A_n}$ has a (nil-Coxeter) presentation: $$ k\left<x_1,\dots,x_{n-1}\right>/\left<x_i^2,x_ix_j-x_jx_i,x_kx_{k+1}x_k - x_{k+1}x_kx_{k+1}\ :\  |i-j|>1,\ 1\leq k\leq n-2\right>. $$
Fix homogeneous coordinates for $\P_d(\E_{A_n})\subseteq \left(\mathbb{P}^{n-2}\right)^{\times {d}}$:
\[ {\bf x} = ([x^0_{1}\colon\cdots\colon x^0_{n-1}],\dots,[x^{d-1}_{1}\colon\cdots\colon x^{d-1}_{n-1}]) \]
and identify $\P_d(\E_{A_n})$ with the zero locus of the homogeneous polynomial system:
\begin{eqnarray*}
&(A_i^r)&\ x_i^{r+1}x_i^r \\ 
&(B_{ij}^r)&\ x_i^{r+1} x_j^r - x_j^{r+1} x_i^r,\\ 
&(C_k^s)&\ x_k^{s+2} x_{k+1}^{s+1} x_k^{s} - x_{k+1}^{s+2} x_k^{s+1} x_{k+1}^{s}
\end{eqnarray*}
for $1\leq i\leq n-1,\ |i-j|>1,\ 1\leq k\leq n-2,\ 0\leq r\leq d-2,\ 0\leq s\leq d-3$.

Fix $d\geq 3$ and fix a point ${\bf x}\in \P_d(\E_{A_n})$. Let $1\leq i\leq n-1$ be any index such that $x_i^0\neq 0$. Then by $(A_i^0)$, $x_i^{1}=0$ and by $(B_{ij}^0)$, $x_j^{1}=0$ for every $j\neq i\pm 1$. If both $x_{i-1}^{1},x_{i+1}^{1}\neq 0$ then by $(A_{i\pm 1}^1),(B_{i\pm 1,j}^1)$, $x_j^{2}=0$ for every $j\neq i$, which forces $x_i^{2}\neq 0$; now by $(C_i^0)$ we obtain $0\neq x_i^{2} x_{i+1}^{1} x_i^{0} = x_{i+1}^{2} x_i^{1} x_{i+1}^{0} = 0$ if $i\neq n-1$ (and $0\neq x_i^{2} x_{i-1}^{1} x_i^{0} = x_{i-1}^{2} x_i^{1} x_{i-1}^{0} = 0$ if $i=n-1\geq 2$, by $(C_{n-2}^0)$), a contradiction. Hence only one coordinate (either $i+1$ or $i-1$) in $[x_1^1\colon\cdots\colon x_{n-1}^1]$ is non-zero. By a similar argument, there exists a sequence $1\leq i_1,\dots,i_{d-1}\leq n-1$ such that for each $r\geq 1$ we have that $x_{i_r}^{r}\neq 0$ and $x_{i'}^{r}=0$ for $i'\neq i_r$, and moreover, $|i_{l+1}-i_l|=1$ and $|i_1-i|=1$. (Notice that for $[x_1^{d-1}\colon\cdots x_{n-1}^{d-1}]$ we can use a similar argument with $(C_{i_{d-3}}^{d-3})$ -- here $i_0=i$.) We claim that either $i<i_1<\cdots<i_{d-1}$ or $i>i_1>\cdots>i_{d-1}$. Otherwise, for some $l$, we have $i_{l+2}=i_l,i_{l+1}=i_l+1$ (or $i_{l+2}=i_l,i_{l+1}=i_l-1$) so by $(C_{i_l}^l)$, 
\[ 0\neq x_{i_{l+2}}^{l+2}x_{i_{l+1}}^{l+1}x_{i_l}^l = x_{i_l}^{l+2}x_{i_l+1}^{l+1}x_{i_l}^l = x_{i_l+1}^{l+2}x_{i_l}^{l+1}x_{i_l+1}^l = 0 \]
(and similarly if $i_{l+1}=i_l-1$).
Finally, notice that $x_{j}^0=0$ for all $j\neq i$. Indeed, by $(A_{i_1}^0)$ we have that $x_{i_1}^0=0$; by $(B_{i_1j}^0)$, $x_j^0=0$ for every $j$ such that $|j-i_1|>1$. Assume that $i_1=i+1$ (otherwise, $i_1=i-1$ and the argument is analogous); it remains to check that $x_{i+2}^0=0$, but otherwise we would have by $(C_{i+1}^0)$ that: 
\[ 0 = x_{i+1}^2x_{i+2}^1x_{i+1}^0 = x_{i+2}^2x_{i+1}^1x_{i+2}^0 = x_{i_2}^2x_{i_1}^1x_{i+2}^0 \neq 0, \]
a contradiction.
For conclusion, we proved that the only points in $\P_d(\E_{A_n})$ are $(e_i,e_{i+1},\dots,e_{i+d-1})$ for $1\leq i\leq n-d$ and $(e_i,e_{i-1},\dots,e_{i-(d-1)})$ for $d\leq i\leq n-1$, where here we denote by $e_i\in \mathbb{P}^{n-2}$ the point defined by $\{x_j=0\ |\ \forall j\neq i\}$. It is also straightforward to check that these points indeed lie in $\P_d(\E_{A_n})$. It follows that:
\[
\#\P_d(\E_{A_n}) = \#\{1,\dots,n-d\} + \#\{d,\dots,n-1\} = 2(n-d)
\]
as long as $3\leq d\leq n-1$ and $0$ if $d\geq n$, as claimed.

For the case of $d=2$, notice that degree-$2$ truncated point modules `see' only quadratic relations, in the sense that they arise from truncated point modules of the maximal quadratic approximation of $\E_{A_n}$:
 \[ k\left<x_1,\dots,x_{n-1}\right>/\left<x_i^2,\ x_ix_j-x_jx_i\ \colon\ |i-j|>1\right>. \] Fix a point ${ \bf x} \in \P_2(\E_{A_n}) \subseteq  \mathbb{P}^{n-2}\times \mathbb{P}^{n-2}$. Arguing as above, if $x_i^0\neq 0$ then $x_j^1=0$ unless $j=i\pm 1$; and if $x_i^1\neq 0$ then $x_j^0=0$ unless $j=i\pm 1$. This shows that ${\bf x}$ must take the form $(e_i,c_ie_{i-1}+c_i'e_{i+1})$ or $(c_ie_{i-1}+c_i'e_{i+1},e_i)$ for each $2\leq i\leq n-2$ (here, by $\alpha e_i + \beta e_j$ we mean the point in $\mathbb{P}^{n-2}$ whose $i$-th coordinate is $\alpha$, whose $j$-th coordinate is $\beta$, and all of whose other coordinates are zero). Furthermore, all such points are indeed in $\P_2(\E_{A_n})$. This gives us $2(n-3)$ projective lines if $n>3$, and if $n=3$, we remain with two points: $(e_1,e_2),(e_2,e_1)$.
\end{proof}

Recall the cyclic relation in $\E_n$ (see \cite[Lemma 7.2]{FK} and \cite{GeneralizedFK}): for $m = 3,\dots,n$ and every $1 \leq a_1 < \cdots < a_m \leq n$:
\begin{align} \label{cyclic relation}
\sum_{i=2}^{m} x_{a_1,a_i}x_{a_1,a_{i+1}}\cdots x_{a_1,a_m}x_{a_1,a_2}\cdots x_{a_1,a_i} = 0.
\end{align}

\begin{prop} \label{prop:star}
    Let $K_{1,n}$ be the star graph with $n+1$ vertices and $n$ edges. Then $\P_d(\E_{K_{1,n}}) = \emptyset$ for $d\geq n+1$.
\end{prop}

\begin{proof}
For simplicity, let us label the vertex at the center of the star by $0$ and the rest of the vertices by $1,\dots,n$. Denote the $n$ generators of $\E_{K_{1,n}}$ (corresponding to the $n$ edges of the star) by $x_1,\dots,x_n$, so $x_i=x_{0,i}$ is the edge connecting the vertices $0$ and $i$.
We argue that $\P_d(\E_{K_{1,n}}) = \emptyset$ whenever $d\geq n+1$; it is evidently sufficient to prove this for $d=n+1$. Otherwise, let:  \[ {\bf x} = (x^0,\dots,x^{d-1})\in \P_d(\E_{K_{1,n}}) \subseteq  \left(\mathbb{P}^{n-1}\right)^{\times d} \]
where $x^i = [x_1^i\colon \cdots \colon x_n^i]$. For each $0\leq i\leq d-1$, let $S_i\subseteq \{1,\dots,n\}$ be the support of $x_i$, namely, $S_i = \{1\leq p\leq n\ |\ x_p^i\neq 0\}$. Suppose that $S_i \cap S_{i+l} \neq \emptyset $ for some $0\leq i < i+l \leq d-1$ with the smallest possible $l>0$ (notice that $l\geq 2$, as all of the generators are squared zero). Pick $t_i,t_{i+1},\dots,t_{i+l}=t_i$ from $S_i,\dots,S_{i+l}$, respectively, such that $t_i,\dots,t_{i+l-1}$ are all distinct. By (\ref{cyclic relation}) for $a_1=0,a_2=t_{i},a_3=t_{i+l-1},\dots,a_{l+1}=t_{i+1}$ (notice that $3 \leq l+1\leq d=n+1$, and $\E_{K_{1,n}}\subseteq \E_{n+1}$), we have that:
\begin{eqnarray} \label{cyclic relation star}
x_{t_i}^{i+l}x_{t_{i+l-1}}^{i+l-1}\cdots x^{i+1}_{t_{i+1}}x^{i}_{t_i} & + & x_{t_{i+l-1}}^{i+l} x_{t_{i+l-2}}^{i+l-1}\cdots x^{i+1}_{t_{i}}x^i_{t_{i+l-1}} + \cdots \\ & + & x_{t_{i+1}}^{i+l} x_{t_{i}}^{i+l-1}\cdots x^{i+1}_{t_{i+2}}x^i_{t_{i+1}} = 0. \nonumber
\end{eqnarray}
Notice that in (\ref{cyclic relation star}), all of the monomials except for the first one vanish, since $t_{i+1},\dots,t_{i+l-1}\notin S_i$; but the first monomial is non-zero, since $t_i\in S_i,\dots t_{i+l-1}\in S_{i+l-1},t_i=t_{i+l}\in S_{i+l}$, a contradiction.

It follows that all of the supports $S_0,S_1,\dots,S_{d-1}$ are pairwise disjoint. Since these supports must be non-empty, it follows that for $d\geq n+1$ we have at least $n+1$ pairwise disjoint non-empty subsets of $\{1,\dots,n\}$, a contradiction; hence $\P_{n+1}(\E_{K_{1,n}})=\emptyset$.
\end{proof}

We can now prove Theorem \ref{thm:trees}. Recall that the line graph $L(G)$ of a graph $G$ is a graph whose vertices are the edges in $G$, and in which two vertices are connected by an edge if the corresponding edges in $G$ share a common vertex.

\begin{proof}[{Proof of Theorem \ref{thm:trees}}]
Let $\T$ be a tree with $n$ edges, say, $x_1,\dots,x_n$. By abuse of notation, let us identify $x_1,\dots,x_n$ with the corresponding degree-$1$ generators of $\E_\T$.
Assume to the contrary that $\E_\T$ admits a truncated point module:
\[
M = ke_0 \oplus \cdots \oplus ke_{n+1}
\]
and pick, for each $0\leq i\leq n$, an edge $x_{a_i}$ such that $x_{a_i}\cdot e_i \neq 0$ (namely, $x_{a_i} \cdot e_i$ is a non-zero scalar multiple of $e_{i+1}$). Thus $a_0,\dots,a_n \in \{1,\dots,n\}$. Furthermore, we claim that for each $0\leq i\leq n-1$, the edges   $x_{a_i},x_{a_{i+1}}$ share a common vertex. For otherwise, their corresponding generators would commute with each other and $ke_i\oplus ke_{i+1}\oplus ke_{i+2}$ would be a truncated point module of degree $2$ (re-grading $\deg(e_i)=0,\deg(e_{i+1})=1,\deg(e_{i+2})=2$) over the commutative algebra $k[x,y]/\left<x^2,y^2\right>$ (via $x\mapsto x_{a_i},\ y\mapsto x_{a_{i+1}}$), whose only truncated point modules are of degree $\leq 1$.
 
By the Pigeonhole Principle, there exist $0\leq i < j\leq n$ such that $a_i = a_j$; we may further assume that $|i-j|$ is smallest possible. (Notice that $j\geq i+2$, since otherwise $x_{a_{i+1}}x_{a_i} \cdot e_i=x_{a_i}^2 \cdot e_i=0$.) Hence the sequence of edges: 
\[
C \colon  x_{a_i},x_{a_{i+1}},\dots,x_{a_{j-1}},x_{a_j} = x_{a_i}
\]
is a cycle (hence a bi-connected subgraph) in the line graph $L(\T)$. Fix a maximal bi-connected component containing $C$; since $L(\T)$ is a block graph (see \cite[Theorem 8.5]{Harary} for more information on line graphs), $C'$ is a clique. In particular, each two edges among $x_{a_i},\dots,x_{a_j}$ share a common vertex in $\T$.

For each $i\leq t\leq j-1$, let $v_t$ denote the (unique) common vertex incidenting both $x_{a_t}$ and $x_{a_{t+1}}$. If $v_i\neq v_{i+1}$ then the edges $x_{a_i},x_{a_{i+1}},x_{a_{i+2}}$ form a triangle in $\T$, contradicting the assumption that $\T$ is a tree, so $v_i = v_{i+1}$. 
Either $v_i=v_{i+1}=\cdots=v_{j-1}$, implying that $x_{a_i},\dots,x_{a_j}$ are edges of a star sub-graph $K_{1,j-i}$ of $\T$ (recall that $x_{a_j}=x_{a_i}$), or else, $v_i=v_{i+1}=\cdots=v_{s-1}\neq v_s$ for some $i+2 \leq s\leq j-1$. Let $w_i,\dots,w_s$ be the endpoints of the edges $x_{a_i},\dots,x_{a_s}$ in $\T$ (respectively) which are \emph{not} $v_i$ and observe that they are all distinct.  
Thus $x_{a_i},\dots,x_{a_s}$ all share a common vertex $v_i$, but $x_{a_{s+1}}$ does not. One of the two vertices incidenting with $x_{a_{s+1}}$ is $w_s$, since $x_{a_s}, x_{a_{s+1}}$ share a common vertex which is not $v_i$; but $x_{a_{s+1}}$ must share a vertex with $x_{a_i},x_{a_{i+1}}$ (as they form a clique in $L(\T)$), so the vertices $w_i,w_{i+1}$ must incident with $x_{a_{s+1}}$ as well. Since $s\geq i+2$, we observe that $w_i,w_{i+1},w_s$ are distinct, a contradiction (they must all incident with $x_{a_{s+1}}$). For conclusion, we proved that $x_{a_i},\dots,x_{a_j}$ form a star graph $K_{1,j-i}$. Therefore, if we re-name the homogeneous basis elements $e_i,\dots,e_{j+1}$ of $M$ by $f_0,\dots,f_{(j-i)+1}$, respectively, and re-grade them by $\deg(f_0)=0,\dots,\deg(f_{(j-i)+1})=j-i+1$ then:
\[
M' = kf_0 \oplus \cdots \oplus ke_{(j-i)+1} = ke_i \oplus \cdots \oplus ke_{j+1}
\]
is a degree-$(j-i+1)$ truncated point module over $\E_{K_{1,j-i}}$, a contradiction to Proposition \ref{prop:star}. This, together with the observation from Proposition \ref{prop:An} that $p(\E_{A_n})=n-1$ (recall that $A_n$ has $n-1$ edges), completes the proof.
\end{proof}

\end{document}